\documentclass[amscd,amssymb,verbatim,10pt]{amsart}
\usepackage[fontsize = 10bp]{fontsize}
\usepackage{scalerel}
\usepackage{stackengine,wasysym}
\usepackage{bbm}
\usepackage{enumerate}

\usepackage[all,cmtip]{xy}

\usepackage{amssymb,latexsym, amsmath, amscd, array, graphicx, pb-diagram}
\usepackage{hyperref}
\usepackage{color}

\usepackage{tikz-cd}

\usepackage{color}
\hypersetup{
    colorlinks=true,
    linkcolor=blue,
    urlcolor=blue,
    citecolor=blue
           }

\usepackage{amsthm}
\usepackage{hyperref}
\usepackage[all]{xy}
\usepackage[scr]{rsfso}

\theoremstyle{plain}

\newtheorem{thm}{Theorem}[section]
\newtheorem{conjecture}[thm]{Conjecture}
\newtheorem{theorem}[thm]{Theorem}

\newtheorem{lemma}[thm]{Lemma}

\newtheorem{proposition}[thm]{Proposition}
\newtheorem{corollary}[thm]{Corollary}

\theoremstyle{definition}

\newtheorem{definition}[thm]{Definition}

\theoremstyle{remark}
\newtheorem{ex}[thm]{Example}
\newtheorem{remark}[thm]{Remark}

\newtheorem{question}[thm]{Question}

\DeclareMathOperator{\cat}{{\mathsf{cat}}}

\DeclareMathOperator{\TC}{{\mathsf{TC}}}

\DeclareMathOperator{\secat}{{\mathsf{secat}}}
\DeclareMathOperator{\dsecat}{{\mathsf{dsecat}}}

\DeclareMathOperator{\dcat}{{\mathsf{dcat}}}
\DeclareMathOperator{\dTC}{{\mathsf{dTC}}}

\DeclareMathOperator{\cd}{{\mathsf{cd}}}

\newcommand{\D}{\mathcal{D}}

  \newcommand{\F}{\mathcal{F}}
\usepackage{enumitem}
\setlist[enumerate]{itemsep=0pt}

\def\id{\protect\operatorname{id}}

\makeatletter
\def\@tocline#1#2#3#4#5#6#7{\relax
  \ifnum #1>\c@tocdepth 
  \else
    \par \addpenalty\@secpenalty\addvspace{#2}%
    \begingroup \hyphenpenalty\@M
    \@ifempty{#4}{%
      \@tempdima\csname r@tocindent\number#1\endcsname\relax
    }{%
      \@tempdima#4\relax
    }%
    \parindent\z@ \leftskip#3\relax \advance\leftskip\@tempdima\relax
    \rightskip\@pnumwidth plus4em \parfillskip-\@pnumwidth
    #5\leavevmode\hskip-\@tempdima
      \ifcase #1
       \or\or \hskip 1em \or \hskip 2em \else \hskip 3em \fi%
      #6\nobreak\relax
    \hfill\hbox to\@pnumwidth{\@tocpagenum{#7}}\par
    \nobreak
    \endgroup
  \fi}
\makeatother

\def\id{\textup{id}}

\def\H{{ H}}

\def\sig{{{\sigma}}}
\def\Sig{{{\Sigma}}}
\def\ta{{{\tau}}}
\def\alp{{{\alpha}}}
\def\bet{{{\beta}}}

\def\T{{ T}}

\def\P{{\mathcal P}}

\def\1{\hbox{\rm\rlap {1}\hskip.03in{\rom I}}}
\def\Bbbone{{\rm1\mathchoice{\kern-0.25em}{\kern-0.25em}
{\kern-0.2em}{\kern-0.2em}I}}

\def\wt{\widetilde}
\def\wh{\widehat}

\usepackage{comment}

\long\def\forget#1\forgotten{} %

\newcommand\ver[1]{\marginpar{\tiny Changed in Ver \VER}}

\begin{document}
\begin{abstract}
We introduce a notion of topological complexity of a group with respect to a family of subgroups, with the case of the trivial and diagonal families recovering the classical category and topological complexity, respectively. These invariants come equipped with a battery of upper and lower bounds derived from functoriality and Bredon cohomology. We introduce a new family of subgroups, the permutational family, and show that the corresponding invariant is a lower bound for the distributional topological complexity of Dranishnikov--Jauhari and Knudsen--Weinberger. As a first application, we show that the permutational and diagonal families coincide, and thus that classical and distributional topological complexity are equal, for a large class of torsion-free groups, extending work of Dranishnikov. Second, we show that the lower bounds of Grant--Lupton--Oprea are in fact lower bounds on distributional topological complexity; in particular, it follows that Farber's conjecture holds distributionally.
\end{abstract}

\title{Motion planning invariants and families of subgroups}

\author{Ekansh~Jauhari}

\address{Ekansh Jauhari, Department of Mathematics, University of Florida, 358 Little Hall, Gainesville, FL 32611, USA.}

\email{ekanshjauhari@ufl.edu}

\author{Ben Knudsen}

\address{Ben Knudsen, Department of Mathematics, Colorado State University, 101 Louis R. Weber Building, Fort Collins, CO 80523, USA.}

\email{b.knudsen@colostate.edu}

\maketitle

\section{Introduction}

Motion planning invariants measure numerically the difficulty of navigation in space---classically, of navigation to the basepoint via the category $\cat$; of navigation with arbitrary destination via the topological complexity $\TC$; and of navigation with prescribed intermediate stops via the sequential topological complexities $\TC_r$, where $\TC_2=\TC$~\cite{Jam, Far, Rud}. More recently, the inevitable issue of discontinuity in motion planning has been reconceived in probabilistic terms, leading to ``distributional'' versions of these classical invariants, distinguished by the prefix $\mathsf{d}$~\cite{DJ,KW1}. 

In the aspherical context, all of the above become group invariants by homotopy invariance, prompting the search for group-theoretic interpretations. To this end, the groundbreaking work of~\cite{FGLO} interpreted $\TC$ in terms of a certain family of subgroups. Inspired by this interpretation, we define a concept of topological complexity $\TC_\F$ with respect to a family of subgroups $\F$ (Definition~\ref{def: main definition}), which we relate to the above invariants.

\begin{theorem}\label{thm: homotopy fixed points}
Let $G$ be a discrete group. We have the following relationships:
\begin{align*}
\dTC_r(G)&\geq \TC_{\P_r}(G^r)\\
\dcat(G)&\geq \TC_{\mathcal{FIN}}(G)\\
\TC_r(G)&=\TC_{\D_r}(G^r)\\
\cat(G)&=\TC_{\{1\}}(G).
\end{align*}
Here, the families $\D_r$ and $\P_r$ are the diagonal and permutational families of Definitions~\ref{def: diagonal family} and~\ref{def: permutational family}, respectively, and $\mathcal{FIN}$ is the family of all finite subgroups.
\end{theorem}

The interest and novelty of this result lie primarily in the first claimed inequality, but it is worth commenting on the others. The third is essentially due to~\cite{FGLO} for $r=2$ and~\cite{FO} in general,\footnote{In these references the result is claimed only for groups whose classifying spaces admit finite models---groups of type $F$, occasionally known as ``geometrically finite'' groups. Since $\TC_r(G)=\infty$ if $G$ is not of type $F$, this restriction is fairly harmless, but the result is true in general, as we prove.} while the fourth is a reformulation of the Eilenberg--Ganea theorem~\cite{EG}. For torsion-free $G$, the second is equivalent to the main result of~\cite{KW1}, and it is trivial for $G$ finite. In the intermediate case of an infinite group with torsion, it is new and potentially interesting.

One avenue of application of Theorem~\ref{thm: homotopy fixed points} lies in investigating the possibility of coincidence between $\TC$ and $\dTC$. Indeed, in view of the string of (in)equalities
\[\TC_{\P_r}(G^r)\leq \dTC_r(G)\leq \TC_r(G)=\TC_{\D_r}(G^r),\]
it is enough to establish the equality of the outer terms. We identify a class of groups, the \emph{impermutable} groups of Section~\ref{section: impermutability}, for which this equality manifestly holds, and investigating the properties of this class leads to the following result, extending work of~\cite{Dr2}.

\begin{corollary}\label{cor: equality impermutable}
For every $r>1$, we have $\dTC_r(G)=\TC_r(G)$ for $G$ torsion-free nilpotent, torsion-free hyperbolic, uniquely divisible, bi-orderable, or residually any of these. 
\end{corollary}

At the time of writing, no torsion-free group $G$ is known for which $\dTC$ and $\TC$ differ, and lists such as Corollary~\ref{cor: equality impermutable} would tend to raise the suspicion that no such group exists. If this suspicion is to be proven valid, however, it will not be for obvious reasons; indeed, the fundamental group of the Klein bottle has the property that $\D_r\neq \P_r$ for every $r>1$. We postpone a detailed consideration of this group to future work.

A second type of application seeks to bound $\dTC$ from below. This effort is facilitated by a battery of general functorial and cohomological bounds established for $\TC_\F$ in Section~\ref{section: invariant}. Using these bounds, we deduce the following.

\begin{theorem}\label{thm: subgroup lower bound}
Fix a natural number $r>1$ and a torsion-free subgroup $K\leq G^r$. If the intersection of $K$ with every conjugate of the diagonal subgroup is trivial, then 
\[\dTC_r(G)\geq \cd(K).\]
\end{theorem}

This result should be compared with~\cite{GLO,FO}, which establish the same inequality for classical (sequential) topological complexity. In other words, Theorem~\ref{thm: subgroup lower bound} shows that this well-known lower bound holds already at the distributional level; in particular, we recover the previously cited theorems as consequences,
as well as the equality $\dTC_r(G)=\TC_r(G)$ in cases of their application~\cite{GLO,GRM,BRM,FO,Kn,Kn2}. For example, it follows that Farber's conjecture on the stable topological complexity of pure graph braid groups is valid distributionally~\cite{Kn}.

Theorem~\ref{thm: homotopy fixed points} also yields a second lower bound on $\dTC$ in terms of what we call the principal component dimension of the family $\P_r$, an invariant defined in terms of Bredon cohomology---see Proposition~\ref{prop: cohomological bounds}. 

It should be emphasized that we do not know whether the first inequality of Theorem~\ref{thm: homotopy fixed points} is ever strict. Since the analogous second inequality is an equality in the torsion-free setting, we formulate the following.

\begin{conjecture}
We have $\dTC_r(G)=\TC_{\P_r}(G^r)$ for every torsion-free group $G$ and $r>1$.
\end{conjecture}

More generally, we pose the following fundamental question.

\begin{question}
Is there a torsion-free group $G$ and $r>1$ such that one of the inequalities  $\TC_{\P_r}(G^r)\leq \dTC_r(G)\leq \TC_r(G)=\TC_{\D_r}(G^r)$ is strict?
\end{question}

\subsection*{Conventions} All groups are discrete, and all topological spaces are compactly generated Hausdorff. Partitions are unordered set partitions. An extended natural number is an element of $\mathbb{Z}_{\geq0}\cup\{\infty\}$. We rely on the general theory of $G$-CW complexes, referred to here as $G$-complexes for brevity, a general reference for which is~\cite{tD}. Given a group $G$ and elements $g,h\in G$, we write $g^h=hgh^{-1}$.

When discussing the probabilistic invariants simultaneously introduced in~\cite{DJ,KW1}, we follow the convention of~\cite{Dr2} in using the word ``distributional'' as in the former reference (rather than ``analog'' as in the latter), while adopting the more general and flexible definition of the latter reference, which topologizes the relevant space of measures via the quotient topology rather than the L\'{e}vy--Prokhorov topology. The two definitions are expected to coincide in general, but it is not known whether they do at the time of writing. 

\subsection*{Acknowledgements} The authors thank Alexander Dranishnikov, Michael Farber, Mark Grant, John Oprea, and Shmuel Weinberger for helpful conversations and correspondence. The second author was supported by NSF grant DMS-2551600.

\section{Invariants from families}

This section is concerned with the definition and basic properties of the invariant $\TC_\F(G)$, the topological complexity of $G$ with respect to the family of subgroups $\F$. After collecting the necessary background material and examples in Sections~\ref{section: families} and~\ref{section: examples}, we define the invariant in Section~\ref{section: invariant} and establish a number of functorial and cohomological bounds on it.

\subsection{Families of subgroups}\label{section: families} We begin by recalling a few basic facts concerning families of subgroups. For more on this topic, we refer the reader to~\cite{Lu} and the references therein.

\begin{definition}\label{def: model}
A set $\F$ of subgroups of $G$ is called a \emph{family} if it is closed under the formation of conjugates and finite intersections.\footnote{Some references require a family to be closed under the formation of subgroups. Following~\cite{Lu,FGLO}, we do not impose this requirement.} We say that a $G$-complex $X$ has \emph{isotropy in} $\F$ if the stabilizer $G_x$ lies in $\F$ for every $x\in X$. We say that $X$ is a \emph{model for (the classifying space of}) $\F$ if
\begin{enumerate}
    \item $X$ has isotropy in $\F$, and
    \item any $G$-complex $Y$ with isotropy in $\F$ admits an equivariant map to $X$, unique up to equivariant homotopy.
\end{enumerate}
\end{definition}

For any family $\F$, a model for $\mathcal{F}$ exists and is unique up to equivariant cellular homotopy equivalence, so it is often safe to refer to \emph{the} classifying space of $\F$, which is denoted generically by $E_{\F}(G)$. Note that, for a subfamily $\F'\subseteq\F$, there is an essentially unique equivariant map $E_{\F'}(G)\to E_{\F}(G)$, which may be taken to be cellular. 

For future use, we recall the following characterization of $E_\F(G)$.

\begin{thm}\label{thm: luck}
    Let $\F$ be a family of subgroups. A $G$-complex $X$ is a model for $\mathcal{F}$ if and only if
    \begin{enumerate}
        \item $X$ has isotropy in $\F$, and
        \item the fixed point set $X^H$ of $X$ is weakly contractible for each $H\in\F$.
    \end{enumerate}
\end{thm}

Although $E_\mathcal{F}(G)$ is a well-defined equivariant homotopy type, the $n$-skeleton $E_\mathcal{F}(G)_n$ is not so \emph{a priori}, so some care is required in dealing with skeleta of classifying spaces, as will be essential for us. 

\begin{lemma}\label{lem: skeleta of models}
Let $\mathcal{F}$ be a family of subgroups of $G$ and $X$ and $Y$ be two models for $\mathcal{F}$. For any $n\geq0$, there are equivariant cellular maps $X_n\to Y_n\to X_n$.
\end{lemma}
\begin{proof}
By our assumption on $X$ and $Y$, there are essentially unique equivariant cellular maps $X\to Y\to X$, and restricting to $n$-skeleta yields the result.
\end{proof}

Given a homomorphism $\varphi:K\to G$ and a family of subgroups $\mathcal{F}$ of $G$, we define $\varphi^*\mathcal{F}=\{\varphi^{-1}(H): H\in\mathcal{F}\}$, a set of subgroups of $K$.

\begin{proposition}\label{prop: restriction family}
The set $\varphi^*\mathcal{F}$ is a family and $\varphi^*E_\F(G)$ is a model for $\varphi^*\F$.
\end{proposition}
\begin{proof}
Closure under intersection is immediate from the fact that the preimages and intersections commute. 
For closure under conjugation, given $H\in \mathcal{F}$ and $k\in K$, the inclusion $\varphi^{-1}(H)^k\subseteq \varphi^{-1}(H^{\varphi(k)})$ is immediate. Since the latter subgroup lies in $\varphi^*\mathcal{F}$, it suffices to demonstrate the reverse inclusion. Given $a\in K$ and $h\in H$ with $\varphi(a)=h^{\varphi(k)}$, rearranging grants that $a^{k^{-1}}\in \varphi^{-1}(H)$, so $a\in \varphi^{-1}(H)^k$, as desired.

For the second claim, since we are dealing with discrete groups, the restriction $\varphi^*E_\F(G)$ is a $K$-CW complex, and we have $K_x=\varphi^{-1}(G_x)$ and $\varphi^*E_\F(G)^{\varphi^{-1}(H)}=E_\F(G)^H$, so the validity of the criteria of Theorem~\ref{thm: luck} for the restriction follows from that of the corresponding criteria for $E_\F(G)$.
\end{proof}

\subsection{Examples of families and models}\label{section: examples} We collect a few examples of models that will be referenced in the remainder of the paper.

\begin{ex}\label{ex: singleton}
    The family of all subgroups of $G$ is modeled by a singleton.
\end{ex}

\begin{ex}\label{ex: trivial}
The family $\{1\}$ comprising only the trivial subgroup is modeled by the universal cover $EG$ of $BG$.
\end{ex}

Since an arbitrary intersection of families is again a family, and since the collection of all subgroups is a family, it is sensible to speak of the family generated by a set of subgroups.

\begin{definition}\label{def: diagonal family}
The $r$th \emph{diagonal family} is the family $\D_r=\D_r(G)$ of subgroups of $G^r$ generated by the trivial subgroup and the image of the diagonal homomorphism $G\to G^r$.
\end{definition}

As shown in~\cite{FGLO,FO}, the diagonal family admits a relatively simple model.

\begin{theorem}[Farber--Grant--Lupton--Oprea, Farber--Oprea]\label{thm: fglo}
The infinite join $(G^{r-1})^{\star \infty}$ is a model for $\D_r$.
\end{theorem}

We write $\mathcal{FIN}=\mathcal{FIN}(G)$ for the collection of all finite subgroups of $G$, which is easily seen to be a family. In the following result, we consider the simplex $\Delta^G$ spanned by $G$ with the affine linear $G$-action induced by left multiplication.

\begin{proposition}\label{prop: simplex over G}
The barycentric subdivision of the simplex $\Delta^{G}$ is a model for $\mathcal{FIN}$.
\end{proposition}

For the proof, we require an observation concerning spaces satisfying the hypotheses of Theorem~\ref{thm: luck} save for the basic assumption of being a $G$-complex. 

\begin{lemma}\label{lem: barycentric}
Let $\mathcal{F}$ be a family of subgroups of $G$ and $X$ a simplicial complex with a simplicial $G$-action. If $X$ satisfies the two conditions of Theorem~\ref{thm: luck}, then the barycentric subdivision of $X$ is a model for $\mathcal{F}$.
\end{lemma}

We will also make use of the following easy fact about affine linear actions on convex spaces.

\begin{lemma}\label{lem: convex}
Let $X\subseteq\mathbb{R}^n$ be a convex subspace with an affine linear $G$-action. For any subgroup $H\leq G$, the fixed set $X^H$ is convex.
\end{lemma}

\begin{proof}[Proof of Proposition~\ref{prop: simplex over G}]
By Lemma~\ref{lem: barycentric}, it suffices to verify the criteria of Theorem~\ref{thm: luck} for $\Delta^G$. For the first criterion, supposing that $g$ fixes $\sum_{i=1}^n t_i g_i$ with the $g_i$ distinct, it follows that $\{g_i\}_{i=1}^n=\{gg_i\}_{i=1}^n$ so $g\in\{g_1g_i^{-1}\}_{i=1}^n$, a finite set. For the second criterion, since every fixed set is convex by Lemma~\ref{lem: convex}, it suffices to note that the fixed set of $H$ is non-empty for finite $H$, since $H$ fixes the barycenter of the face spanned by $H$. 
\end{proof}

For the proof of the lemma, we recall the following characterization of $G$-complexes~\cite[Rmk. 1.3]{Lu}.

\begin{proposition}\label{prop: G-CW criterion}
A CW complex $X$ with a cellular $G$-action is a $G$-complex if and only if, for every $g\in G$, any open cell $e$ fixed setwise by $g$ is fixed pointwise by $g$.
\end{proposition}

\begin{corollary}\label{cor: barycentric}
Let $X$ be a simplicial complex with a simplicial $G$-action. The barycentric subdivision of $X$ with its canonical $G$-action is a $G$-complex. 
\end{corollary}
\begin{proof}
An open $n$-simplex $\sigma$ of the barycentric subdivision is indexed by a chain $\sigma_0\subsetneq \cdots\subsetneq \sigma_n$ of simplices of $X$, and $g\sigma$ is the simplex indexed by $g\sigma_0\subsetneq\cdots\subsetneq g\sigma_n$. Assuming that $g\sigma=\sigma$, it follows that $g\sigma_i=\sigma_i$ for each $0\leq i\leq n$, so $g$ fixes the vertices of $\sigma$. Since the action is simplicial, it follows that $g$ fixes $\sigma$ pointwise, and Proposition~\ref{prop: G-CW criterion} applies.
\end{proof}

Lemma~\ref{lem: barycentric} now follows immediately from Theorem~\ref{thm: luck} and Corollary~\ref{cor: barycentric}.

\subsection{Topological complexity with respect to a family}\label{section: invariant} We come now to the main definition. Recall that the space of homotopy fixed points of $G$ acting
on $X$ is the space of equivariant maps $X^{hG}=\mathrm{Map}^G(EG, X)$. 

\begin{definition}\label{def: main definition}
Let $G$ be a group and $\F$ a family of subgroups of $G$. The \emph{$\F$-topological complexity of $G$} is the quantity
\[\TC_\F(G)=\min\left\{n\ge 0\ \middle| \ E_{\F}(G)_{n}^{hG} \ne \varnothing \right\}.\]
\end{definition}

Note that $\TC_\F(G)$ is well-defined by Lemma~\ref{lem: skeleta of models}. The following alternate characterization is sometimes useful.

\begin{proposition}\label{prop: atc char-2}
Let $\F$ be a family of subgroups containing the trivial subgroup. The quantity $\TC_\F(G)$ is the least extended natural number $n$ for which the canonical map $EG=E_{\{1\}}(G)\to E_\F(G)$ factors through $E_\F(G)_n$ up to equivariant homotopy.
\end{proposition}
\begin{proof}
By definition, $\TC_\F(G)\leq n$ if and only if there is an equivariant map $EG\to E_\F(G)_n$. The result of composing this map with the inclusion into $E_\F(G)$ is homotopic to the canonical map by Definition~\ref{def: model}.
\end{proof}

In some cases, the invariant $\TC_\F(G)$ recovers a well-known quantity. 

\begin{ex}
When $\F$ is the family of all subgroups, we have $\TC_\F(G)=0$ by Example~\ref{ex: singleton}.
\end{ex}

\begin{ex}\label{example: trivial family cd}
In view of Example~\ref{ex: trivial}, we have $\TC_{\{1\}}(G)=\cd(G)$, the cohomological dimension of $G$. Indeed, by Example~\ref{ex: trivial} and Proposition~\ref{prop: atc char-2}, the left-hand quantity is the least $n$ for which the $n$-skeleton of $EG$ is a homotopy retract. That this latter number is the cohomological dimension is a well-known fact, which may be proven by the same argument as~\cite[Thm. 7.4]{KW1}, for example.
\end{ex}

An advantage of the generality of Definition~\ref{def: main definition} is that it gives rise to a number of flexible bounds. First, we have bounds arising from functoriality in the variables $\F$ and $G$.

\begin{proposition}\label{prop: functorial bounds}
Let $\mathcal{F}$ be a family of subgroups of $G$.
\begin{enumerate}
\item For any subfamily $\mathcal{F}'\subseteq \mathcal{F}$, we have $\TC_\F(G)\leq \TC_{\F'}(G)$.
\item For any homomorphism $\varphi: K\to G$, we have $\TC_{\varphi^*\mathcal{F}}(K)\leq \TC_\mathcal{F}(G)$.
\end{enumerate}
\end{proposition}
\begin{proof}
For the first claim, setting $n=\TC_{\mathcal{F}'}(G)$, the definition grants the existence of an equivariant map $EG\to E_{\mathcal{F}'}(G)_n$. Composing this map with the restriction to $n$-skeleta of the essentially unique map $E_{\mathcal{F}'}(G)\to E_{\mathcal{F}}(G)$, we conclude that $\TC_\F(G)\leq n$, as desired. For the second claim, setting $n=\TC_\F(G)$, the definition grants the existence of a $G$-equivariant map $EG\to E_\F(G)_n$. Composing with the natural map from $EK$,
we obtain the $K$-equivariant map
$EK\to \varphi^*EG\to \varphi^*E_\F(G)_n,$ so $\TC_{\varphi^*\F}(K)\leq n$ by Proposition~\ref{prop: restriction family}, as desired.
\end{proof}

\begin{corollary}\label{cor: general subgroup lower bound}
Let $\F$ be a family of subgroups of $G$ and $K\leq G$ a subgroup. If $K\cap H=\{1\}$ for every $H\in \F$, then
\[\TC_\F(G)\geq \cd(K).\]
\end{corollary}
\begin{proof}
Writing $i:K\to G$ for the inclusion, the assumption grants that $i^*\F=\{1\}$, so the conclusion follows from Example~\ref{example: trivial family cd} and Proposition~\ref{prop: functorial bounds}.
\end{proof}

The second type of bound relies on the theory of Bredon cohomology (see~\cite{Bre,Mis}, for example). As we need very little of this theory, we will be very brief.

Fixing a group $G$ and a family $\F$ of subgroups thereof, we recall that a coefficient system is a presheaf of Abelian groups on the orbit category $\mathcal{O}_\F$ of $\F$, which has objects transitive $G$-sets with isotropy in $\F$ and morphisms equivariant functions. Given a $G$-complex $X$ with isotropy in $\F$ and a coefficient system $M$, the Bredon cohomology $H^*_\F(X;M)$ is defined. In the case $\F=\{1\}$, we recover ordinary equivariant cohomology $H^*_G$, i.e., the cohomology of the Borel construction. In addition to functoriality in the space variable, Bredon cohomology enjoys contravariant functoriality for inclusions of families. Thus, if $1\in\F$, there arises a restriction homomorphism $H^*_\F(X;M)\to H^*_G(X;M(G))$. In the case $X=EG$, the target of this homomorphism is simply group cohomology.

\begin{definition}
Let $\F$ be a family of subgroups of $G$.
\begin{enumerate}
\item The \emph{geometric dimension} of $G$ with respect to $\F$ is the least extended natural number $\mathsf{gd}_\F(G)$ such that $\F$ admits a model of dimension $\mathsf{gd}_\F(G)$.
\item The \emph{cohomological dimension} of $G$ with respect to $\F$ is the least extended natural number $\cd_\F(G)$ such that $H^n_\F(G;M)=0$ for every coefficient system $M$ and $n>\cd_\F(G)$.
\item The \emph{principal component dimension} of $G$ with respect to $\F$ is the least extended natural number $\mathsf{pcd}_\F(G)$ such that the restriction homomorphism $H^n_\F(EG;M)\to H^n(G;M(G))$ induced by the inclusion of the trivial family vanishes for every coefficient system $M$ and $n>\mathsf{pcd}_\F(M)$. If $1\notin\F$, then $\mathsf{pcd}_\F(G)$ is undefined.
\end{enumerate}
\end{definition}

\begin{proposition}\label{prop: cohomological bounds}
For any group $G$ and family of subgroups $\F$ containing the trivial subgroup, we have the inequalities
\[\mathsf{pcd}_\F(G)\leq  \TC_\F(G)\leq \mathsf{gd}_\F(G)\leq \max \{3, \cd_\F(G)\}.\]
\end{proposition}
\begin{proof}
The rightmost inequality is~\cite[Theorem~5.2]{Lu}. For the middle inequality, choosing a model for $\F$ of dimension $n=\mathsf{gd}_\F(G)$, the assumption that $1\in\F$ grants the existence of an essentially unique equivariant cellular map $EG=E_{\{1\}}(G)\to E_\F(G)=E_\F(G)_n$, as desired. For the leftmost inequality, set $n=\mathsf{pcd}_\F(G)$, and suppose that $\TC_\F(G)<n$. From the definitions and Proposition~\ref{prop: atc char-2}, there is a coefficient system $M$ for which the restriction homomorphism $H^n_\F(G;M)\to H^n_\F(EG;M)$ is nontrivial and a commutative diagram of the form
\[
\xymatrix{
H^n_\F(EG;M)   & & 
H^n_{\F}(E_{\F}(G); M) \ar[ll] \ar[dl]
\\
& H^n_{\F}(E_{\F}(G)_{n-1}; M).\ar[ul] 
}
\]
Since $E_{\F}(G)_{n-1}$ is $(n-1)$-dimensional, the group at the bottom of this diagram vanishes, so the top homomorphism is trivial by commutativity, a contradiction.
\end{proof}

\section{The permutational families}

This section introduces the key player in our study, namely family of subgroups $\P_r$ of $G^r$. The main results are Proposition~\ref{prop: family}, verifying that $\P_r$ is in fact a family, and Theorem~\ref{thm: classifying model}, which identifies the model for $\P_r$ subsequently used in the proof of Theorem~\ref{thm: homotopy fixed points}.

\subsection{Defining the families}

We fix a natural number $r>1$, which we suppress whenever possible. As a matter of notation, given $g\in G$, we write $\wh g=(g,\ldots,g)\in G^{r-1}$. Throughout, we regard $G^{r-1}$ as a $G^r$-set by writing $G^r=G\times G^{r-1}$ and defining 
\[(g,h)k=\wh gkh^{-1}.\]

\begin{lemma}\label{lem: cosets}
There is a canonical isomorphism of $G^r$-sets of the form
\[G^r/G\cong G^{r-1}.\]
\end{lemma}
\begin{proof}
By orbit-stabilizer, it suffices to show that the action of $G^r$ on $G^{r-1}$ is transitive, and that the stabilizer of $\widehat 1\in G^{r-1}$ is the diagonal subgroup. For transitivity, we have $(1, g_1^{-1},\ldots, g_{r-1}^{-1})\widehat 1=(g_1,\ldots, g_{r-1}),$ while requiring that $(g, h_1,\ldots, h_{r-1})$ stabilize $\widehat 1$ is equivalent to requiring that $gh_i^{-1}=1$ for each $i$.
\end{proof}

Our construction of the permutational families will rely on the following auxiliary concept.

\begin{definition}\label{def: type sigma}
    Fix subset $T\subseteq G^{r-1}$ and a permutation $\sigma\in \Sigma_{T}$. We say that an element $g\in G$ is of \emph{type $\sigma$} if
    \[
    \sigma(x)\sigma(y)^{-1}\wh g=\wh gxy^{-1}
    \]
    for all $x,y\in T$.
\end{definition}

We write $Z(T,\sigma)\subseteq G$ for the set of all elements of $G$ of type $\sigma$. Given $g\in Z(T,\sigma)$, it is immediate from the definition that $\sigma(x)^{-1}\widehat gx$ is independent of $x\in T$, and we write $g_\sigma$ for this common value. We set 
\[
\Delta(T,\sigma):=\left\{ \left(g,g_\sigma\right) \mid  g\in Z(T,\sigma) \right\}\subseteq G^{r}.
\]Given a subgroup $H\leq \Sigma_T$, we write $\Delta(T,H)=\bigcup_{\sigma\in H}\Delta(T,\sigma)$. Regarding this set, we have the following observation.

\begin{proposition}\label{prop: subgroup}
    For any subset $T\subseteq G^{r-1}$ and subgroup $H\leq \Sigma_{T}$, the subset $\Delta(\T,\H)\subseteq G^{r}$ is a subgroup.
\end{proposition}
\begin{proof}
Given  $ g$ of type $\sigma$ and  $h$ of type $\tau$, a straightforward verification shows that $gh$ is of type $\sigma\tau$ and that $g_\sigma h_\tau=(gh)_{\sigma\tau}$. It is likewise straightforward to verify that $g^{-1}$ is of type $\sigma^{-1}$ and that $(g_\sigma)^{-1}=(g^{-1})_{\sigma^{-1}}$. The claim now follows upon observing that the set in question is nonempty, since $1$ is of type $\id$ and $1_\mathrm{id}=1$.
\end{proof}

The example of greatest interest for us is the case in which $H$ is the subgroup $\Sigma_\alpha\leq \Sigma_T$ preserving the blocks of a partition $\alpha$ of $T$. After choosing an ordering $T\cong\{1,\ldots,n\}$, the subgroups of the form $\Sigma_\alpha$ are precisely the conjugates of Young subgroups of $\Sigma_{n}$.

\begin{definition}\label{def: permutational family}
The $r$th \emph{permutational family} $\P_r=\P_r(G)$ is the set comprising the trivial subgroup and all subgroups of $G^{r}$ of the form $\Delta(\T,\Sig_{\alp})$, where $T$ ranges over finite subsets of $G^{r-1}$ and $\alpha$ over partitions of $T$.
\end{definition}

Our language is justified by the following result.

\begin{proposition}\label{prop: family}
For any $r>1$, the set $\P_r$ is a family of subgroups.
\end{proposition}

For the proof of this result, we begin with the following lemma.

\begin{lemma}\label{lem: intersection}
Given subsets $S,T\subseteq G^{r-1}$ and permutations $\sigma\in \Sigma_{S}$ and $\tau\in\Sigma_{T}$, we have
\[
\Delta(S,\sig)\cap\Delta(T,\ta)=
     \begin{cases}
\Delta(S\cup T,\sig\cup\ta) & \text{ if }\sigma|_{S\cap T}=\tau|_{S\cap T}
        \\
        \varnothing & \text{ otherwise},
    \end{cases}   
\]
where $\sigma\cup\tau\in\Sigma_{S\cup T}$ is the unique permutation on $S\cup T$ restricting to $\sigma$ on $S$ and to $\tau$ on $T$.  
\end{lemma}

\begin{proof}
Given $x\in S\cap T$, any element of the intersection in question has the form 
\[(g, \sigma(x)^{-1}\wh gx)=(g,\tau(x)^{-1}\wh g x),\] and rearranging yields the equation $\sigma(x)=\tau(x)$. Thus, we may assume that $\sigma|_{S\cap T}=\tau|_{S\cap T}$, writing  $\rho=\sigma\cup\tau$ for brevity. In this case, an element $g$ is of type $\rho$ if and only if $g$ is both of type $\sigma$ and of type $\tau$, and $(g,g_\sigma)=(g,g_\rho)=(g,g_\tau)$ for such $g$, implying the claim.
\end{proof}

Recall that the \emph{meet} of the partitions $\alpha$ and $\beta$ of a common set $X$, denoted $\alpha \wedge \beta$, is the partition of $X$ formed by the collection of all non-empty intersections of blocks of $\alpha$ with blocks of $\beta$; equivalently, the meet is the greatest lower bound of $\alpha$ and $\beta$ in the lattice of partitions. If $\alpha$ and $\beta$ are partitions of the respective subsets $S$ and $T$ of $X$, we abusively write $\alpha\wedge \beta$ for the meet of $\alpha\cup\{X\setminus S\}$ and $\beta\cup\{X\setminus T\}$.

\begin{lemma}\label{lem: meet inclusion}
Fix a set $X$, finite subsets $S,T\subseteq X$ with respective partitions $\alpha$ and $\beta$, and permutations $\sigma\in \Sigma_S$ and $\tau\in \Sigma_T$. If $\sigma|_{S\cap T}=\tau|_{S\cap T}$, then $\sigma\cup\tau\in \Sigma_{\alpha\wedge\beta}$.
\end{lemma}
\begin{proof}
The claim is that $\sigma\cup \tau$ preserves the blocks of $\alpha\wedge\beta$. Now, these blocks are of three types: first, blocks of $\alpha$ lying in $S\setminus T$; second, blocks of $\beta$ lying in $T\setminus S$; and, third, intersections of blocks of $\alpha$ and blocks of $\beta$ lying in $S\cap T$. In the first case, $\sigma\cup\tau$ acts as $\sigma$, which preserves the blocks of $\alpha$ and preserves $S\setminus T$ by assumption. The second case is similar. In the third case, $\sigma\cup \tau$ acts equally as $\sigma$ and as $\tau$, hence preserves the blocks of $\alpha$ and of $\beta$, and preserves $S\cap T$ by assumption.
\end{proof}

\begin{corollary}\label{cor: intersection meet}
Given finite subsets $S,T\subseteq G^r$ with respective partitions $\alpha$ and $\beta$, we have $\Delta(S,\Sigma_\alpha)\cap \Delta(T,\Sigma_\beta)=\Delta(S\cup T, \Sigma_{\alpha\wedge \beta})$.
\end{corollary}
\begin{proof}
The rightward inclusion follows from Lemmas~\ref{lem: intersection} and~\ref{lem: meet inclusion}, in view of the equality
\[
\Delta(S,\Sig_{\alp})\cap \Delta(T,\Sig_{\bet})=\bigcup_{\substack{\sig\in\Sig_{\alp}\\ \ta\in\Sig_{\bet}}} \Delta(S,\sig)\cap \Delta(T,\ta).
\] On the other hand, writing $\overline \alpha=\alpha\cup\{T\setminus S\}$, we have $\Sigma_{\alpha\wedge\beta}\leq \Sigma_{\overline\alpha}$ and $\Delta(S\cup T,\Sigma_{\overline\alpha})\subseteq\Delta(S, \Sigma_\alpha)$, implying the reverse inclusion.
\end{proof}

\begin{proof}[Proof of Proposition~\ref{prop: family}] Closure under intersection is immediate from Corollary~\ref{cor: intersection meet}. As for conjugation, fix a finite subset $T\subseteq G^{r-1}$, a permutation $\sigma\in \Sigma_T$, and an arbitrary element $(h,k)\in G\times G^{r-1}$. Defining $\wt T=\wh hTk^{-1}$, we obtain a permutation $\wt\sigma\in \Sigma_{\wt T}$ by the formula $\wt \sigma(\wh hxk^{-1})=\wh h\sigma(x)k^{-1}$. With these definitions, a direct verification shows that $g$ is of type $\sigma$ if and only if $g^h$ is of type $\wt \sigma$, and that $(g,g_\sigma)^{(h,k)}=(g^h, (g^h)_{\wt \sigma})$, whence $\Delta(T,\sigma)^{(h,k)}=\Delta(\wt T,\wt \sigma)$. In the same way, a partition $\alpha$ of $T$ determines a partition $\wt \alpha$ of $\wt T$, and it follows that $\Delta(T,\Sigma_{\alpha})^{(h,k)}=\Delta(\wt T,\Sigma_{\wt \alpha})\in \P_r$, as desired.
\end{proof}

To conclude, we clarify the relationship between $\P_r$ and the earlier defined diagonal family $\D_r$.

\begin{proposition}\label{prop: diagonal equivalence}
A subgroup of $G^r$ lies in $\D_r$ if and only if it is of the form $\Delta(T,\id)$ for some finite subset $T\subseteq G^{r-1}$. In particular, the family $\D_r$ is a subfamily of $\P_r$.
\end{proposition}
\begin{proof}
It is shown in~\cite{FGLO} for $r=2$ and~\cite{FO} in general that a non-trivial subgroup of $G^{r}$ lies in $\D_r$ if and only if it is of the form \[
H_{x,S}:=\left\{\left(g,x\wh gx^{-1}\right)\ \middle| \ \wh g\in Z(S) \right\} \leq G^{r}.
\] for some $x\in G^{r-1}$ and finite $S\subseteq G^{r-1}$. For the ``if direction'', choose a finite set $T\subseteq G^{r-1}$, and set $S=\{xy^{-1}\mid x,y\in T\}$. Then we have $Z(T,\id)=\{g\in G\mid \widehat g\in Z(S)\}$ and $\Delta(T,\id)=H_{x,S}$ by inspection. For the ``only if'' direction, it follows from Lemma~\ref{lem: intersection} and the proof of Proposition~\ref{prop: family} concerning conjugates that the set of subgroups of the form $\Delta(T,\id)$ is a family of subgroups. Thus, to establish that every member of $\D_r$ lies in this family, it suffices to show that the diagonal subgroup does so, but the diagonal subgroup is $\Delta(\{\,\widehat 1\,\}, \id)$ by inspection.
\end{proof}

\subsection{A model for $\P_r$} In this section, we consider the simplex $\Delta^{G^{r-1}}$ spanned by the set $G^{r-1}$ with the $G^r$-action induced by the action on the set of vertices defined in the previous section. The main technical result of the paper is the following.

\begin{theorem}\label{thm: classifying model}
The barycentric subdivision of the simplex $\Delta^{G^{r-1}}$ is a model for $\P_r$.
\end{theorem}

This result naturally provokes the question of whether and when the simplex itself is a model for $\P_r$. We answer this question below in Proposition~\ref{prop: cw failure}.

For the proof, we begin by noting that a point $\mu=\sum_{x\in X}t_xx$ of the simplex spanned by a set $X$ has two natural combinatorial invariants associated to it. First, there is the support of $\mu$, which is the subset of $x\in X$ for which $t_x\neq0$. Second, $\mu$ determines a partition of its support, whose blocks are the subsets of elements with equal coefficients.

\begin{lemma}\label{lem: measure stabilizer}
For $\mu\in \Delta^{G^{r-1}}$ with support $T$ and associated partition $\alpha$, the stabilizer subgroup of $\mu$ under the action of $G^{r}$ is $\Delta(\T,\Sig_{\alp})$.
\end{lemma}
\begin{proof}
We begin by writing $\mu=\sum_{x\in T}t_xx$. Given $\sigma\in \Sigma_\alpha$ and $g$ of type $\sigma$, we recall that $g_\sigma=\sigma(x)^{-1}\wh gx$ for any $x\in T$, so 
\begin{align*}
(g,g_\sigma)\cdot \mu &= \sum_{x\in T}t_x \bigl(\wh gx\bigl(\sigma(x)^{-1}\wh gx\bigr)^{-1}\bigr)
=\sum_{x\in T}t_x \sigma(x)=\mu,
\end{align*}
where the last equality follows from our assumption on $\sigma$. It follows that $\Delta(\T,\Sig_{\alp})$ stabilizes $\mu$. Conversely, supposing that $(g,h)\in G\times G^{r-1}$ stabilizes $\mu$, it follows that, for each $x\in T$, we have $\wh gxh^{-1}=z$ for some $z\in T$ with $t_x=t_z$. Thus, the formula $\sigma(x)=\wh gx h^{-1}$ defines a permutation $\sigma\in \Sigma_\alpha$. Moreover, we have
\[
\sigma(x)\sigma(y)^{-1}\wh g=\bigl(\wh gxh^{-1}\bigr)\bigl(hy^{-1}\wh g^{-1}\bigr)\wh g=\wh gxy^{-1},
    \] for $x,y\in T$, so $g$ has type $\sigma$, and rearranging the definition shows that $h=g_\sigma$, so $(g,h)\in \Delta(T,\Sigma_\alpha)$, as desired.
\end{proof}

\begin{lemma}\label{lem: designer measure}
For any finite subset $T\subseteq G^{r-1}$ and partition $\alpha$, there is a point $\mu\in\Delta^{G^{r-1}}$ with support $T$ and associated partition $\alpha$.
\end{lemma}
\begin{proof}
Write $\alpha=\{A_i\}_{i=1}^\ell$ with $|A_i|$ non-increasing. Since $\sum_{i=1}^\ell i=\frac{\ell(\ell+1)}{2}$, the expression 
$\mu=\sum_{i=1}^\ell \sum_{x\in A_i} \frac{2i}{\ell(\ell+1)|A_i|} x$ defines an element of $\Delta^{G^r}$ with support $T$. To see that $\alpha$ is the partition associated to $\mu$, note that, for $i<j$, we have 
\[\frac{2i}{\ell(\ell+1)|A_i|}<\frac{2j}{\ell(\ell+1)|A_i|}\leq \frac{2j}{\ell(\ell+1)|A_j|},\] since $|A_i|$ is non-increasing. In particular, elements of $T$ have common coefficient in $\mu$ if and only if they lie in the same block of $\alpha$, as claimed.
\end{proof}

\begin{proof}[Proof of Theorem~\ref{thm: classifying model}]
By Lemma~\ref{lem: barycentric}, it suffices to verify the criteria of Theorem~\ref{thm: luck}. Lemma~\ref{lem: measure stabilizer} implies that $\Delta^{G^{r-1}}$ has isotropy in $\P_r$. Since the same result together with Lemma~\ref{lem: designer measure} implies that the fixed point set $(\Delta^{G^{r-1}})^H$ is non-empty for every $H\in\P_r$, the proof is complete upon noting that this set is also convex by Lemma~\ref{lem: convex}.
\end{proof}

\section{Classical and probabilistic invariants}

In this section, after recalling the definitions of the invariants involved, we prove Theorems~\ref{thm: homotopy fixed points} and~\ref{thm: subgroup lower bound}. We then introduce and study the notion of an impermutable group, a type of group for which classical and distributional topological complexities coincide for ``obvious'' reasons.

\subsection{Definitions and basic properties} We begin by recalling the definition of sectional category, from which subsequent definitions will be specialized. In keeping with modern usage, we adopt the ``reduced'' convention for such invariants, under which, for example, the category of a singleton is $0$. The reader is advised that certain, especially older, sections of the literature employ an alternate convention differing from ours by $1$.

\begin{definition}\label{def: secat}
Let $f:X\to Y$ be a map of topological spaces. The \emph{sectional category} of $f$, denoted $\secat(f)$, is the least $n$ for which there is an open cover $\{U_i\}_{i=0}^n$ of $Y$ such that $f$ admits a local section over each $U_i$.
\end{definition}

In most cases, this definition has an equivalent and very useful formulation. As a matter of notation, given a map $f:X\to Y$, we write $X^{\star m}_Y$ for the $m$-fold fiberwise join of $X$ over $Y$---see~\cite{Sch, Jam} for details, as well as for the proof of the following result.

\begin{lemma}\label{lem: join secat}
Let $p:E\to B$ be a Serre fibration. If $B$ is paracompact, then $\secat(p)\leq n$ if and only if the induced map $E^{\star(n+1)}_B \to B$ admits a section.
\end{lemma}

A point of the fiberwise join is an ordered formal convex combination of points lying in a single fiber of the map in question. Such a sum defines a probability measure on this fiber, which motivates an alternative conception of sectional category~\cite{DJ,KW1}. As a matter of notation, we write $\mathcal{B}_m(X)$ for the space of probability measures on $X$ with support of cardinality at most $m$. For more on the topology of these spaces, we direct the reader to~\cite{KK,KW1}.\footnote{As explained above, we work with the more general and flexible quotient topology on the space of probability measures, as in~\cite{KW1}, rather than the more geometric and narrow L\'{e}vy--Prokhorov topology, as in~\cite{DJ}. Nevertheless, we adopt the adjective ``distributional'' of the latter reference.}

\begin{ex}\label{ex: simplex}
If $X$ is discrete, then $\mathcal{B}_m(X)\cong\Delta^X_{m-1}$, the $(m-1)$-skeleton of the simplex spanned by $X$.
\end{ex}

Given a map $f:X\to Y$, we similarly write $\mathcal{B}_m(f)\subseteq \mathcal{B}_m(X)$ for the subspace of measures with support lying in a single fiber of $f$.

\begin{definition}
Let $f:X\to Y$ be a map. The \emph{distributional sectional category} of $f$, denoted $\dsecat(f)$, is the least $n$ for which the induced map $\mathcal{B}_{n+1}(f)\to Y$ admits a section.
\end{definition}

In view of the map from the fiberwise join to the space of probability measures alluded to above, we have the inequality $\dsecat(f)\leq \secat(f)$. 

\begin{lemma}\label{lem: secat homotopy}
Consider the following commutative diagram of maps between topological spaces:
\[
\xymatrix{
E_1\ar[dr]_-{p_1}\ar[rr]&&E_2\ar[dl]^-{p_2}\\
&B.
}
\] If the diagonal maps are Hurewicz fibrations and the horizontal map a homotopy equivalence, then $\secat(p_1)=\secat(p_2)$ and $\dsecat(p_2)=\dsecat(p_2)$.
\end{lemma}
\begin{proof}
The first claim follows from \cite[Prop.~2.1]{Rud}, and the second is a special case of~\cite[Prop. 5.2]{KW1}. 
\end{proof}

We come now to the main definitions in our inquiry. For more on these invariants, we direct the reader to~\cite{Jam,Far,BGRT,DJ,KW1,Ja}.

\begin{definition}
Let $G$ be a group. We write $\pi^1_G:BG^{[0,1]}\to BG$ for the map given by evaluation at $0$. For $r>1$, we write $\pi_G^r:BG^{[0,1]}\to BG^r$ for the map with $i$th coordinate given by evaluation at $\frac{i-1}{r-1}$. 
\begin{enumerate}
\item The \emph{category} of $G$ is $\cat(G)=\secat(\pi^1_G)$.
\item The \emph{distributional category} of $G$ is $\dcat(G)=\dsecat(\pi^1_G)$.
\item The $r$th \emph{sequential topological complexity} of $G$ is $\TC_r(G)=\secat(\pi_G^r)$.
\item The $r$th \emph{sequential distributional topological complexity} of $G$ is $\dTC_r(G)=\dsecat(\pi_G^r)$.
\end{enumerate}
\end{definition}

\subsection{Proof of Theorems~\ref{thm: homotopy fixed points} and~\ref{thm: subgroup lower bound}} The first step in the proof is to replace the path space with something more manageable.

\begin{lemma}\label{lem: path space borel}
For any $r>1$, there is a commutative diagram of the following form, in which the horizontal arrow is a homotopy equivalence:
\[\xymatrix{
BG^{[0,1]}\ar[rr]\ar[dr]_-{\pi_G^r}&&EG^r/G\ar[dl]\\
&BG^r.
}\]
\end{lemma}
\begin{proof}
Consider the following commutative solid diagram of maps between topological spaces:
\[\xymatrix{
EG\ar[d]\ar[r]&EG^r\ar[d]\\
BG\ar[d]_-i\ar[r]&EG^r/G\ar[d]\\
BG^{[0,1]}\ar@{-->}[ur]\ar[r]^-{\pi_G^r}&BG^r.
}\] Here, we have written $i$ for the inclusion of the constant paths, the top arrow is the diagonal map, which is $G$-equivariant, and the middle horizontal arrow is the map induced on $G$-orbits. Now, the lower right-hand vertical map is a fiber bundle, hence a Hurewicz fibration, and, since $BG$ is a CW complex, the map $i$ is a trivial Hurewicz cofibration; therefore, the indicated dashed filler exists. Since the top arrow is a weak equivalence with free action on source and target, the induced map $BG\to EG^r/G$ is also a weak equivalence, hence a homotopy equivalence by the Whitehead theorem. The claim now follows by two-out-of-three.
\end{proof}

The second step is to understand the local structure of the maps in question.

\begin{lemma}\label{lem: fiber bundles}
Let $p:E\to B$ be a fiber bundle with fiber $F$ and structure group $K$, and fix $n>0$.
\begin{enumerate}
\item The induced map $E_B^{\star n}\to B$ is a fiber bundle with fiber $F^{\star n}$ and structure group $K$.
\item The induced map $\mathcal{B}_n(p)\to B$ is a fiber bundle with fiber $\mathcal{B}_n(F)$ and structure group $K$, provided $B$ is a CW complex and $F$ locally compact.
\end{enumerate}
More specifically, if $\{k_{\alpha\beta}\}$ is a $K$-cocycle for $p$ defined on the trivializing cover $\{U_\alpha\}$, then it is also a $K$-cocycle for each of the resulting fiber bundles.
\end{lemma}
\begin{proof}
The first claim is essentially~\cite[Prop. 1]{Sch}, and the second follows from~\cite[Cor. 4.8, Lem. 4.10]{KW1}.
\end{proof}

Finally, we record the following well-known relationship between sections and homotopy fixed points. For a proof, the reader may consult~\cite[Lem. 2.4]{KW2}, for example. 

\begin{lemma}\label{lem: section space}
For any group $K$ and $K$-space $X$, the homotopy fixed point space $X^{hK}$ is canonically weakly equivalent to the space of sections of the canonical map $EK\times_K X\to BK$.
\end{lemma}

With these results in hand, we are now in a position to prove the main results.

\begin{proof}[Proof of Theorem~\ref{thm: homotopy fixed points}]
For the first inequality, Lemmas~\ref{lem: secat homotopy} and~\ref{lem: path space borel} imply that $\dTC_r(G)$ is the distributional sectional category of the projection from $EG^r/G$ to $BG^r$, which, via Lemma~\ref{lem: cosets}, we may identify with the projection from the Borel construction $EG^r\times_{G^r} G^{r-1}$. Appealing to Example~\ref{ex: simplex} and Lemma~\ref{lem: fiber bundles}, we find that $\dTC_r(G)$ is the minimal $n$ for which the bundle $EG^r\times_{G^r} \Delta^{G^{r-1}}_{n}$ over $BG^r$ admits a section---here, we use that $\mathcal{B}_{n+1}(X)\cong \Delta^X_{n}$ for $X$ discrete. By Lemma~\ref{lem: section space}, the existence of such a section is equivalent to non-vacuity of $(\Delta_n^{G^{r-1}})^{hG^r}$, and the claim now follows from Theorem~\ref{thm: classifying model}, since $\Delta^{G^{r-1}}$ maps to its barycentric subdivision in a cellular and equivariant manner.

The proof of the equality $\TC_r(G)=\TC_{\D_r}(G^r)$ follows this argument very closely. The same lemmas permit us to deal with the Borel construction rather than the path space, and, since the base space is a CW complex, hence paracompact, Lemmas~\ref{lem: join secat} and~\ref{lem: fiber bundles} imply that the number in question is the minimal $n$ for which the bundle $EG^r\times_{G^r} (G^{r-1})^{\star (n+1)}$ over $BG^r$ admits a section. By Lemma~\ref{lem: section space} and Theorem~\ref{thm: fglo}, it now suffices to argue that $(G^{r-1})^{\star (n+1)}$ admits a homotopy fixed point if and only if $(G^{r-1})^{\star\infty}_n$ does so. The ``only if'' direction is immediate, since the former space is canonically and equivariantly a subspace of the latter, while the ``if'' direction follows from the equivariant Whitehead theorem as in~\cite[4.11]{FO}

The proof of the second inequality follows in the same manner from~\cite[Prop. 2.1]{KW2} and Proposition~\ref{prop: simplex over G}. Finally, in view of Example~\ref{example: trivial family cd}, the equality $\cat(G)=\TC_{\{1\}}(G)$ is equivalent to the Eilenberg--Ganea theorem~\cite{EG}.
\end{proof}

\begin{proof}[Proof of Theorem~\ref{thm: subgroup lower bound}]
By Theorem~\ref{thm: homotopy fixed points} and Corollary~\ref{cor: general subgroup lower bound}, it suffices to show that $K$ contains no elements of the form $(g,g_\sigma)$ with $1\neq \sigma\in \Sigma_T$ for $T\subseteq G^{r-1}$ a finite subset. Supposing otherwise, setting $n=|\sigma|$, and appealing to the proof of Proposition~\ref{prop: subgroup}, we find that $(g,g_\sigma)^n=(g^n, (g^n)_{\id})$ which lies in a conjugate of the diagonal by Proposition~\ref{prop: diagonal equivalence}. It follows that $(g,g_\sigma)^n$ is trivial by our assumption on $K$, so $g^n=1$, a contradiction.
\end{proof}

\subsection{Impermutability}\label{section: impermutability} In this section, we identify a sufficient condition on a group $G$ under which Theorem~\ref{thm: homotopy fixed points} guarantees that $\dTC=\TC$.

\begin{definition}
We say that the group $G$ is $r$-\emph{permutable} for some $r>1$ if there is a finite subset $T\subseteq G^{r-1}$, a nontrivial permutation $\sigma\in \Sigma_T$, and an element $g\in G$ of type $\sigma$. We say that $G$ is \emph{impermutable} if $G$ is not $r$-permutable for some $r>1$.
\end{definition}

As we now show, these conditions are equivalent for varying $r$.

\begin{lemma}\label{lem: permutability}
The group $G$ is $r$-permutable for some $r>1$ if and only if $r$-permutable for every $r>1$.
\end{lemma}
\begin{proof}
It suffices to show, for fixed $r$, that $G$ is $r$-permutable if and only if $2$-permutable. 

Assuming $2$-permutability, and given $g$ of nontrivial type $\sigma\in \Sigma_T$ with $T\subseteq G$, let $\wh T=\{\wh x\mid x\in T\}\subseteq G^{r-1}$. Via the diagonal homomorphism, the sets $T$ and $\wh T$ are canonically in bijection, so $\sigma$ determines a nontrivial permutation $\wh \sigma$ of $\wh T$, and an immediate verification shows that $g$ is of type $\wh \sigma$.

Assume instead that $G$ is $r$-permutable, and choose $g$ of nontrivial type $\sigma\in \Sigma_T$ with $T\subseteq G^{r-1}$. For $x\in T$, let $n(x)$ be the least natural number such that the first coordinates $\sigma^{n(x)}(x)_1$ and $x_1$ coincide. Note that $n(x)$ is bounded above by $|\sigma|$, and, without loss of generality, since $\sigma$ is nontrivial, we have $n(x)>1$ for some $x\in T$. Among all elements of $T$ satisfying this condition, choose $x$ with $n(x)$ minimal. By minimality, the first coordinates $\sigma^{j-1}(x)_1$ are distinct for $1\leq j\leq n(x)$. Writing $S\subseteq G$ for the set of these first coordinates, a direct verification shows that $g$ is of type $\rho$, where $\rho\in \Sigma_S\cong\Sigma_{n(x)}$ is the standard $n(x)$-cycle. Since $n(x)>1$, the claim follows.
\end{proof}

The relevance of the concept of impermutability for our purposes lies in the following result.

\begin{proposition}\label{prop: impermutable families}
If $G$ is impermutable, then $\D_r=\P_r$ for every $r>1$.
\end{proposition}
\begin{proof}
The claim is immediate from Proposition~\ref{prop: diagonal equivalence} and Lemma~\ref{lem: permutability}.
\end{proof}

Appealing to Theorem~\ref{thm: homotopy fixed points}, we draw the following conclusion.

\begin{corollary}\label{cor: impermutable consequence}
If $G$ is impermutable, then $\TC_r(G)=\dTC_r(G)$ for every $r>1$.
\end{corollary}

We now give a number of partial characterizations of impermutability. Recall that a group $G$ is said to be Frobenius injective if the $k$th power map $g\mapsto g^k$ is injective for every $k>0$.

\begin{proposition}\label{prop: impermutable properties}
Let $G$ be a group.
\begin{enumerate}
\item If $G$ is impermutable, then $G$ is torsion-free.
\item If $G$ is a subgroup of an impermutable group, then $G$ is impermutable.
\item If $G$ is a finite product of impermutable groups, then $G$ is impermutable.
\item If $G$ is residually impermutable, then $G$ is impermutable.
\item If $G$ is Frobenius injective, then $G$ is impermutable. 
\end{enumerate}
\end{proposition}
\begin{proof}
For the first claim, given $g\in G$ with finite order $n$, set $T=\{1, g, \ldots, g^{n-1}\}$. A direct verification shows that $g$ has type $\rho$, where $\rho\in \Sigma_T\cong\Sigma_n$ is the standard $n$ cycle. Since $G$ is impermutable, it follows that $n=1$, which is to say that $g=1$, so $G$ is torsion-free.

The second claim is essentially immediate, and the third may be proved by adapting the argument of Lemma~\ref{lem: permutability}. For the fourth, fix a finite subset $T\subseteq G$ and $g\in G$ of type $\sigma\in \Sigma_T$. For each $x\neq y\in T$, our assumption guarantees a homomorphism $\varphi_{x,y}:G\to K_{x,y}$ such that $K_{x,y}$ is impermutable and $\varphi_{x,y}(xy^{-1})\neq 1$. Writing $K=\prod_{x\neq y\in T} K_{x,y}$, we obtain a homomorphism $\varphi: G\to K$ such that $\varphi|_T$ is injective. Thus, we may view $\sigma$ as a permutation of $\varphi(T)$, and it follows that $\varphi(g)$ has type $\sigma$. Since $K$ is impermutable by (3), it follows that $\sigma$ is trivial.

The fifth claim is essentially due to~\cite{Dr2}, but we give an argument here nevertheless. Fix $T\subseteq G$, a permutation $\sigma\in \Sigma_T$, and an element $g\in G$ of type $\sigma$. Without loss of generality, $\sigma$ is an $n$-cycle, and we write $T=\{x_1,\ldots, x_n\}$ accordingly. By definition, we have $g_\sigma=x_{j+1}^{-1} g x_j$ for any $j$, where the indices are taken modulo $n$. Multiplying these $n$ many distinct but equal expressions in the appropriate order, we conclude that \[(x_{2}^{-1}gx_1)^n=g_\sigma^n=x_1^{-1}g^nx_1=(x_1^{-1}g x_1)^n.\] Invoking Frobenius injectivity and rearranging, we find that $x_1=x_2$, implying that $n=1$ and $\sigma$ is trivial. 
\end{proof}

Recall that a group is said to be bi-orderable if it admits a total order invariant under left and right multiplication.

\begin{corollary}\label{cor: impermutable list}
The class of impermutable groups contains the classes of torsion-free nilpotent, torsion-free hyperbolic, uniquely divisible, and bi-orderable groups, as well as the class of groups residually in any of these classes.
\end{corollary}
\begin{proof}
As explained in~\cite{Dr2}, the first three types of groups are Frobenius injective, so it remains to show that bi-orderable groups are also Frobenius injective. Given $x,y\in G$ with $G$ bi-orderable, we may assume that $x<y$, and an easy induction shows that $x^n<y^n$ for every $n>0$, so $x^n\neq y^n$ by irreflexivity.
\end{proof}

Corollary~\ref{cor: equality impermutable} now follows by combining Corollaries~\ref{cor: impermutable consequence} and~\ref{cor: impermutable list}. In view of Proposition~\ref{prop: impermutable properties}, we record an example of a torsion-free group that is not impermutable.

\begin{proposition}\label{prop: klein}
The fundamental group of the Klein bottle is not impermutable.
\end{proposition}
\begin{proof}
Recall that the group in question is $K=\langle a,b\mid a^b=a^{-1}\rangle$. Setting $T=\{a,a^{-1}\}$ and writing $\sigma\in\Sigma_T$ for the unique nontrivial permutation, we claim that $b$ is of type $\sigma$. Indeed, we have
\begin{align*}
\sigma(a)\sigma(a^{-1})^{-1}b&=a^{-2}b
=ba^{2}
=ba(a^{-1})^{-1}\\
\sigma(a^{-1})\sigma(a)^{-1}b&=a^{2}b
=ba^{-2}
=ba^{-1}a^{-1},
\end{align*}
as required by the definition.
\end{proof}

Stronger, one can show that $\D_r(K)\neq \P_r(K)$ for every $r>1$. We defer this verification to future work, in which we undertake a detailed study of these families. Nevertheless, we can give a negative answer to~\cite[Prob. 3.2.6]{Dr2}, which asks whether $\Delta^K$ is a model for $\D_2(K)$. In fact, according to Proposition~\ref{prop: klein} and the following result, this space is not even a $K$-complex.

\begin{proposition}\label{prop: cw failure}
Let $G$ be a group. For any $r>1$, the simplex $\Delta^{G^{r-1}}$ is a $G^r$-complex if and only if $G$ is impermutable.
\end{proposition}
\begin{proof}
Suppose there exist a finite subset $T\subseteq G^{r-1}$, a nontrivial permutation $\sigma\in \Sigma_T$, and an element $g\in G$ of type $\sigma$. Then $(\hat g, g_\sigma)$ stabilizes the barycenter of the open simplex spanned by $T$, but it does not fix this simplex pointwise, since it permutes its vertices nontrivially. The conclusion now follows from Proposition~\ref{prop: G-CW criterion}. Conversely, suppose that $G$ is impermutable and that $(g,h)\cdot \mu$ lies in the same open simplex as $\mu$, say the simplex spanned by $T\subseteq G^{r-1}$. Writing $\mu=\sum_{x\in T}t_x x$, we have
\begin{align*}
(g,h)\cdot \mu &= \sum_{x\in T}t_x (\wh gxh^{-1})
=\sum_{x\in T}s_x x.
\end{align*} Without loss of generality, the barycentric coordinates $t_x$ and $s_x$ are all nonzero, and it follows that the formula $\sigma(x)=\wh gxh^{-1}$ defines a permutation $\sigma\in \Sigma_T$. A direct verification shows that $g$ has type $\sigma$, whence $\sigma=\id$ by impermutability; in particular, it follows that $(g,h)$ fixes the vertices of the simplex spanned by $T$, hence the entire simplex. The conclusion now follows from Proposition~\ref{prop: G-CW criterion}.
\end{proof}

In particular, the $G^{r-1}$-simplex is never a model for any family unless $G$ is impermutable, in which case it is a model for $\D_r$.

\begin{remark}
We do not know whether the converses to Proposition~\ref{prop: impermutable families}, Corollary~\ref{cor: impermutable consequence}, or Proposition~\ref{prop: impermutable properties}(5) hold.
\end{remark}

\bigskip
\section*{}
\footnotesize{\textit{2020 Mathematics Subject Classification.} Primary 
55P91,  	
55M30,      
20F65,  	
Secondary 
55R37,  	
20B99,  	
20J06.  	

\textit{Key words and phrases.} Family of subgroups, classifying space, $G$-complex, topological complexity, simplex spanned by $G$, impermutable group.}
\bigskip

\end{document}